\documentclass[11pt]{amsart}

\usepackage{enumerate}
\usepackage{amsmath,amsthm,verbatim,amssymb,amsfonts,amscd,amsopn,amsxtra,graphicx,lmodern}
\usepackage{hyperref}
\usepackage{tikz-cd} 
\usepackage{graphics}
\usepackage{mathrsfs}
\usepackage{relsize}
\usepackage{enumitem}
\usepackage[utf8]{inputenc}
\usepackage{csquotes}
\usepackage{amsthm}
\usepackage{thmtools}
\usepackage{mathtools}
\usepackage{microtype}
\usepackage{pdfrender,xcolor}
\usepackage[sort cites=true, backend=biber]{biblatex}
\usepackage{biblatex}

\tikzset{
	symbol/.style={
		draw=none,
		every to/.append style={
			edge node={node [sloped, allow upside down, auto=false]{$#1$}}}
	}
}

\begin{document}
	\pdfrender{StrokeColor=black,TextRenderingMode=2,LineWidth=0.2pt}	
	
	\title{Minimal pairs, inertia degrees, ramification degrees and implicit constant fields}

	\author{Arpan Dutta}

	\def\NZQ{\mathbb}               
	\def\NN{{\NZQ N}}
	\def\QQ{{\NZQ Q}}
	\def\ZZ{{\NZQ Z}}
	\def\RR{{\NZQ R}}
	\def\CC{{\NZQ C}}
	\def\AA{{\NZQ A}}
	\def\BB{{\NZQ B}}
	\def\PP{{\NZQ P}}
	\def\FF{{\NZQ F}}
	\def\GG{{\NZQ G}}
	\def\HH{{\NZQ H}}
	\def\UU{{\NZQ U}}
	\def\P{\mathcal P}
	
	%
	%
	\let\union=\cup
	\let\sect=\cap
	\let\dirsum=\oplus
	\let\tensor=\otimes
	\let\iso=\cong
	\let\Union=\bigcup
	\let\Sect=\bigcap
	\let\Dirsum=\bigoplus
	\let\Tensor=\bigotimes
	
	\theoremstyle{plain}
	\newtheorem{Theorem}{Theorem}[section]
	\newtheorem{Lemma}[Theorem]{Lemma}
	\newtheorem{Corollary}[Theorem]{Corollary}
	\newtheorem{Proposition}[Theorem]{Proposition}
	\newtheorem{Problem}[Theorem]{}
	\newtheorem{Conjecture}[Theorem]{Conjecture}
	\newtheorem{Question}[Theorem]{Question}
	
	\theoremstyle{definition}
	\newtheorem{Example}[Theorem]{Example}
	\newtheorem{Examples}[Theorem]{Examples}
	\newtheorem{Definition}[Theorem]{Definition}
	
	\theoremstyle{remark}
	\newtheorem{Remark}[Theorem]{Remark}
	\newtheorem{Remarks}[Theorem]{Remarks}
	
	\newcommand{\trdeg}{\mbox{\rm trdeg}\,}
	\newcommand{\rr}{\mbox{\rm rat rk}\,}
	\newcommand{\sep}{\mathrm{sep}}
	\newcommand{\ac}{\mathrm{ac}}
	\newcommand{\ins}{\mathrm{ins}}
	\newcommand{\res}{\mathrm{res}}
	\newcommand{\Gal}{\mathrm{Gal}\,}
	\newcommand{\ch}{\mathrm{char}\,}
	\newcommand{\Aut}{\mathrm{Aut}\,}
	\newcommand{\kras}{\mathrm{kras}\,}
	\newcommand{\dist}{\mathrm{dist}\,}
	\newcommand{\ord}{\mathrm{ord}\,}
	
	\newcommand{\n}{\par\noindent}
	\newcommand{\nn}{\par\vskip2pt\noindent}
	\newcommand{\sn}{\par\smallskip\noindent}
	\newcommand{\mn}{\par\medskip\noindent}
	\newcommand{\bn}{\par\bigskip\noindent}
	\newcommand{\pars}{\par\smallskip}
	\newcommand{\parm}{\par\medskip}
	\newcommand{\parb}{\par\bigskip}
	\let\epsilon\varepsilon
	\let\phi=\varphi
	\let\kappa=\varkappa
	
	\def \a {\alpha}
	\def \b {\beta}
	\def \s {\sigma}
	\def \d {\delta}
	\def \g {\gamma}
	\def \o {\omega}
	\def \l {\lambda}
	\def \th {\theta}
	\def \D {\Delta}
	\def \G {\Gamma}
	\def \O {\Omega}
	\def \L {\Lambda}
	%
	%
	\textwidth=15cm \textheight=22cm \topmargin=0.5cm
	\oddsidemargin=0.5cm \evensidemargin=0.5cm \pagestyle{plain}

	\address{Department of Mathematics, IISER Mohali,
		Knowledge City, Sector 81, Manauli PO,
		SAS Nagar, Punjab, India, 140306.}
	\email{arpan.cmi@gmail.com}
	
	\date{\today}
	
	\thanks{This work was supported by the Post-Doctoral Fellowship of the National Board of Higher Mathematics, India.}
	
	
	
	\maketitle


\begin{abstract}
	An extension $(K(X)|K,v)$ of valued fields is said to be valuation transcendental if we have equality in the Abhyankar inequality. Minimal pairs of definition are fundamental objects in the investigation of valuation transcendental extensions. In this article, we associate a uniquely determined positive integer with a valuation transcendental extension. This integer is defined via a chosen minimal pair of definition, but it is later shown to be independent of the choice. Further, we show that this integer encodes important information regarding the implicit constant field of the extension $(K(X)|K,v)$. 
\end{abstract}


\section{Introduction} Throughout this article we will assume that $(\overline{K}(X)|K,v)$ is an extension of valued fields, where $\overline{K}$ is a fixed algebraic closure of $K$ and $X$ is an indeterminate. The extension $(K(X)|K,v)$ satisfies the famous Abhyankar inequality:
\begin{equation}\label{eqn Abh inequality}
	\rr vK(X)/vK + \trdeg[K(X)v:Kv] \leq 1,
\end{equation}
where $vK$ and $Kv$ denote respectively the value group and residue field of $(K,v)$ and $\rr vK(X)/vK$ is the $\QQ$-dimension of the divisible hull $\QQ\tensor_{\ZZ} (vK(X)/vK)$. The above inequality is a consequence of [\ref{Bourbaki}, Chapter VI, \S 10.3, Theorem 1]. The extension $(K(X)|K,v)$ is said to be \textbf{valuation transcendental} if we have equality in (\ref{eqn Abh inequality}). The extension is said to be \textbf{value transcendental} if we have $\rr vK(X)/vK = 1$ and \textbf{residue transcendental} if $\trdeg[K(X)v:Kv]=1$. Throughout this article, we will assume that the extension $(K(X)|K,v)$ is valuation transcendental. 

\pars Minimal pairs of definition have been used with great success in the study of valuation transcendental extensions [cf. \ref{AP sur une classe}, \ref{APZ characterization of residual trans extns}, \ref{APZ2 minimal pairs}, \ref{Dutta min fields implicit const fields}]. A pair $(a,\g) \in \overline{K}\times v\overline{K}(X)$ is said to be a \textbf{minimal pair of definition for $v$ over $K$} if it satisfies the following conditions:
\sn (MP1) $v(X-a) = \g = \max v(X-\overline{K})$,
\n (MP2) $v(a-b)\geq \g \Longrightarrow [K(b):K] \geq [K(a):K]$ for all $b\in\overline{K}$, \\
where 
\[ v(X-\overline{K}):= \{ v(X-c) \mid c\in\overline{K} \}. \]
A valuation transcendental extension always admits a minimal pair of definition. It has been observed in [\ref{Kuh value groups residue fields rational fn fields}, Theorem 3.11] that $(K(X)|K,v)$ is value transcendental if and only if $\g\notin v\overline{K}$, that is, if and only if $\g$ is not a torsion element modulo $vK$. Further, it follows from [\ref{Kuh value groups residue fields rational fn fields}, Lemma 3.3] that $(K(X)|K,v)$ is value (residue) transcendental if and only if $(L(X)|L,v)$ is also value (residue) transcendental, where $L$ is an arbitrary algebraic extension of $K$.

\pars The goal of this article is twofold:
\begin{itemize}
	\item associate a uniquely determined positive integer with a valuation transcendental extension,
	\item show that the said integer encodes important information regarding the implicit constant field of the extension (defined later).
\end{itemize}
We first prove the following result in Section \ref{Sec proof of Thm 1.1}:

\begin{Theorem}\label{Thm central}
	Take a minimal pair of definition $(a,\g)$ for $v$ over $K$. Then
	\begin{equation}
		(vK(a,X):vK(X))[K(a,X)v:K(X)v] = j,
	\end{equation} 
	where $v(a-a_i) \geq \g$ for exactly $j$ many conjugates $a_i$ of $a$ over $K$, including $a$ itself and counting multiplicities. 
\end{Theorem}  
Given a minimal pair of definition $(a,\g)$ for $v$ over $K$, we will denote the integer $j$ as defined in Theorem \ref{Thm central} by $j(a,K,\g)$. When $K$ and $\g$ are implicitly understood, we will simply denote it by $j(a)$. We mention here that when the valuation $v$ is induced by a pseudo monotone sequence $M$ in $(K,v)$, the integer $j(a)$ is also referred to as the dominating degree of the minimal polynomial of $a$ over $K$ with respect to $M$ [cf. \ref{Peruginelli, Spirito - extend valns pseudo monotone}, \ref{Dutta rank implicit const pms}]. 

\pars For any residue transcendental extension $(K(X)|K,v)$ with a minimal pair of definition $(a,\g)$, we construct a value transcendental extension $w$ of $v$ to $\overline{K}(X)$ with a minimal pair of definition $(a,\G)$ such that $j(a,K,\g) = j(a,K,\G)$. This observation is then used to obtain the following result:  
\begin{Theorem}\label{Thm independence of j}
	Take minimal pairs of definition $(a,\g)$ and $(a^\prime,\g)$ for $v$ over $K$. Then
	\[ j(a) = j(a^\prime). \]
\end{Theorem}
Theorem \ref{Thm independence of j} illustrates that the integer $j(a)$ depends solely on the valued extension $(\overline{K}(X)|K,v)$ and is independent of the choice of the minimal pair of definition for $v$ over $K$. Hence the notation $j(v,K)$ would be more suited as it reflects the independence of $j$. However, for the sake of continuity we persist with the notation $j(a,K,\g)$ for the remainder of the article. 

\pars Theorem \ref{Thm central} is then applied to the computation of the implicit constant field of the extension $(K(X)|K,v)$. Given an extension of $v$ to $\overline{K(X)}$, the \textbf{implicit constant field} of the extension $(K(X)|K,v)$ is defined as 
\[ IC(K(X)|K,v) := \overline{K} \sect K(X)^h, \]
where $K(X)^h$ is the henselization of $K(X)$ [cf. Section \ref{Sec prelim}]. Implicit constant fields were introduced by Kuhlmann in [\ref{Kuh value groups residue fields rational fn fields}] to construct extensions of $v$ to $K(X)$ with prescribed value groups and residue fields. The problem of the explicit computation of implicit constant fields for valuation transcendental extensions was considered in [\ref{Dutta min fields implicit const fields}], where it was studied via minimal pairs of definition. Given a minimal pair of definition $(a,\g)$ for $v$ over $K$, it has been observed in [\ref{Dutta min fields implicit const fields}, Theorem 1.1] that 
\[ IC(K(X)|K,v) \subseteq K(a)^h. \]
Under the additional assumptions that $a$ is separable over $K$ and there is a unique extension of $v$ from $K$ to $K(a)$, we observe in [\ref{Dutta min fields implicit const fields}, Theorem 1.3] that 
\[ IC(K(X)|K,v) \subsetneq K(a)^h \text{ whenever } \g \leq \kras(a,K), \]
where 
\[ \kras(a,K):= \max \{ v(a-\s a) \mid \s\in \Gal(\overline{K}|K) \text{ and } \s a \neq a \}. \]
In this article, we observe that $j(a)$ is a more natural candidate for the investigation of $IC(K(X)|K,v)$ than $\kras(a,K)$. Specifically, in Theorem \ref{Thm j(a) divides deg K(a)h over IC} we show that 
\[ j(a) \text{ divides } [K(a)^h : IC(K(X)|K,v)]. \]
This observation is independent of the separability of $a$ and also does not depend on the number of extensions of $v$ from $K$ to $K(a)$. When $(K,v)$ is a defectless field [cf. Section \ref{Sec prelim}], then we obtain that 
\[ j(a) = [K(a)^h : IC(K(X)|K,v)]. \]

\pars Given an extension of $v$ to $\overline{K(X)}$, we further observe in [\ref{Dutta min fields implicit const fields}, Theorem 1.1] that 
\begin{equation}\label{eqn IC when v.t.}
	(K(a)\sect K^r)^h \subseteq IC(K(X)|K,v)
\end{equation}
when $v$ is value transcendental, and 
\begin{equation}\label{eqn IC when r.t.}
	(K(a)\sect K^i)^h \subseteq IC(K(X)|K,v)
\end{equation}
when $v$ is residue transcendental, where $K^r$ and $K^i$ denote the absolute ramification field and absolute inertia field of $(K,v)$ [cf. Section \ref{Sec prelim}]. The fact that $K^i \subseteq K^r$ implies that (\ref{eqn IC when v.t.}) gives a tighter bound that (\ref{eqn IC when r.t.}). It is a natural question to inquire whether (\ref{eqn IC when v.t.}) also holds when $v$ is residue transcendental. We prove the following result using Theorem \ref{Thm central}:

\begin{Theorem}\label{Thm IC bounds}
	Take a minimal pair of definition $(a,\g)$ for $v$ over $K$ and fix an extension of $v$ to $\overline{K(X)}$. Then
	\[ (K(a)\sect K^r)^h \subseteq IC(K(X)|K,v) \subseteq K(a)^h. \]
\end{Theorem}


\section{Preliminaries}\label{Sec prelim}

Throughout this article, we denote the value of an element $a$ by $va$ and its residue by $av$. The valuation ring of a valued field $(K,v)$ will be denoted by $\mathcal{O}_K$. The compositum of two fields $k_1$ and $k_2$ contained in some overfield $\O$ will be denoted by $k_1. k_2$. 

\pars Fixing an extension $w$ of $v$ to $\overline{K}$, we define the following distinguished groups:
\begin{align*}
	G^{d(w)} &:= \{ \s \in \Gal(\overline{K}| K) \mid w\circ\s = w \text{ on } K^\sep \},\\
	G^{i(w)} &:= \{ \s \in \Gal(\overline{K}| K) \mid w(\s a - a) > 0 \text{ for all } a\in \mathcal{O}_{K^\sep} \},\\
	G^{r(w)} &:= \{ \s \in \Gal(\overline{K}| K) \mid w(\s a - a) > wa \text{ for all } a\in K^\sep \setminus \{0\} \},
\end{align*}
where $K^\sep$ denotes the separable-algebraic closure of $K$. The corresponding fixed fields in $K^\sep$ will be denoted by $K^{d(w)}$, $K^{i(w)}$ and $K^{r(w)}$ and they are called the \textbf{absolute decomposition field}, \textbf{absolute inertia field} and the \textbf{absolute ramification field} of $(K,v)$. If the extension $w$ is clear from context, we will simply write them as $K^d$, $K^i$ and $K^r$. We have the following chain of inclusions: 
\[ K^d \subseteq K^i \subseteq K^r. \]

\pars A valued field $(K,v)$ is said to be \textbf{henselian} if $v$ admits a unique extension to $\overline{K}$. Every valued field has a minimal separable-algebraic extension which is henselian. This extension is unique up to valuation preserving isomorphisms over $K$, and we can consider it to be the same as the absolute decomposition field. We will call this extension the \textbf{henselization} of $(K,v)$. Clearly, the henselization depends on the choice of the extension of $v$ to $\overline{K}$. When the extension of $v$ to $\overline{K}$ is tacitly understood, we will denote the henselization by $K^h$. The henselization of $(K,v)$ with respect to an extension $w$ of $v$ to $\overline{K}$ will be denoted by $K^{h(w)}$. 

\pars Henselization is an \textbf{immediate} extension, that is, $vK^h = vK$ and $K^hv = Kv$. An algebraic extension of henselian valued fields is again henselian. For any algebraic extension $L$ of $K$, we have that $L^h = L.K^h$.  

\pars For a valued field $(K,v)$ admitting a unique extension of $v$ to a finite extension $L$, we have the \textbf{Lemma of Ostrowski} which states that
\[ [L:K] = (vL:vK)[Lv:Kv] p^d \text{ for some }d\in\NN, \]
where $p:= \ch Kv$ when $\ch Kv>0$ and $p:=1$ otherwise. The number $p^d$ is said to be the \textbf{defect} of the extension $(L|K,v)$ and will be denoted by $d(L|K,v)$. The extension $(L|K,v)$ is said to be \textbf{defectless} if $d(L|K,v) = 1$. Defect satisfies the following \textbf{multiplicative property}: if $L|F|K$ is a tower of fields such that $L|K$ is finite and $v$ admits a unique extension from $K$ to $L$, then 
\[ d(L|K,v) = d(L|F,v) d(F|K,v). \]
An arbitrary algebraic extension $(\O|K,v)$ is said to be defectless if $d(L|K,v) = 1$ for every finite subextension $L|K$. 

\pars Observe that the Lemma of Ostrowski is applicable in particular to henselian valued fields. It is well-known that $(K^r|K,v)$ is a defectless extension for a henselian valued field $(K,v)$. A henselian field $(K,v)$ is said to be defectless if $d(L|K,v)=1$ for every finite extension $(L|K,v)$. We will say that an arbitrary valued field is defectless if its henselization is defectless.


\section{Proof of Theorem \ref{Thm central}}\label{Sec proof of Thm 1.1}

\begin{Lemma}\label{Lemma d = vg}
	Take a minimal pair of definition $(a,\g)$ for $v$ over $K$. Then for any $\d\in vK(a)$, there exists a polynomial $g(X)\in K[X]$ with $\deg g < [K(a):K]$ such that $vg = \d$.
\end{Lemma}

\begin{proof}
	Set $n:= [K(a):K]$. The fact that $\d\in vK(a)$ implies that we can take $c_i \in K$ such that $\d = v \sum_{i=0}^{n-1}c_i a^i$. Define $g(X):= \sum_{i=0}^{n-1} c_i X^i \in K[X]$. Then $\deg g < [K(a):K]$. It now follows from [\ref{APZ characterization of residual trans extns}, Theorem 2.1] and [\ref{Dutta min fields implicit const fields}, Lemma 3.2] that $\d = v g(a) = vg$.
\end{proof}

\begin{Lemma}\label{Lemma K(a,X)v}
	Assume that $(K(X)|K,v)$ is residue transcendental. Take a minimal pair of definition $(a,\g)$ for $v$ over $K$. Let $E$ be the least positive integer such that $E\g\in vK(a)$. Take $f(X)\in K[X]$ with $\deg f < [K(a):K]$ such that $vf = -E\g$. Then 
	\[ K(a,X)v = K(a)v (f(X)(X-a)^Ev). \]

\end{Lemma}

\begin{proof}
	Observe that $(K(a,X)|K(a),v)$ is a residue transcendental extension with minimal pair of definition $(a,\g)$. Take $d\in K(a)$ such that $vd = -E\g$. It follows from [\ref{APZ characterization of residual trans extns}, Theorem 2.1] that $K(a,X)v = K(a)v (d(X-a)^Ev)$. By Lemma \ref{Lemma d = vg}, we can take $f(X)\in K[X]$ with $\deg f < [K(a):K]$ such that $vf = vd$. It now follows from [\ref{Dutta min fields implicit const fields}, Lemma 6.1] that $\frac{d}{f}v \in K(a)v$. As a consequence, 
	\[ K(a,X)v = K(a)v (d(X-a)^E v) = K(a)v \mathlarger{\mathlarger{(}}(\frac{d}{f}v) (f(X)(X-a)^Ev)\mathlarger{\mathlarger{)}} = K(a)v (f(X)(X-a)^Ev). \]
\end{proof}

\pars We can now give a \textbf{proof of Theorem \ref{Thm central}}.

\begin{proof}
	Take the minimal polynomial $Q(X)$ of $a$ over $K$. Write
	\[ Q(X) = (X-a)(X-a_2) \dotsc (X-a_j) \dotsc (X-a_n), \]
	where $v(a-a_i) \geq \g$ for $2\leq i\leq j$ and $v(a-a_i)  <\g$ for all $i>j$. By definition, the assumption that $(a,\g)$ is a pair of definition for $v$ over $K$ implies that $v(X-b) = \min \{ \g, v(a-b) \}$ for all $b\in\overline{K}$. It follows that
	\[ vQ = j\g + \a, \text{ where } \a:= v(a-a_{j+1}) + \dotsc v(a-a_n) \in v\overline{K}. \]
	
	\pars We first assume that $(K(X)|K,v)$ is a value transcendental extension, that is, $\g$ is not a torsion element modulo $vK$. It follows from [\ref{Dutta min fields implicit const fields}, Remark 3.3] that 
	\[ vK(X) = vK(a) \dirsum \ZZ(j\g+\a) \text{ and } K(X)v = K(a)v. \]
	Further, we observe that $(a,\g)$ is a minimal pair of definition for $v$ over $K(a)$. Consequently, 
	\[ vK(a,X) = vK(a) \dirsum \ZZ\g \text{ and } K(a,X)v = K(a)v. \]
	The facts that $j\g+\a \in vK(a)\dirsum \ZZ\g$ and $\g$ is not a torsion element modulo $vK$ imply that $\a\in vK(a)$. Consequently, $vK(X) = vK(a)\dirsum \ZZ j \g$. It follows that $(vK(a,X):vK(X)) = j$ and hence
	\[ (vK(a,X):vK(X))[K(a,X)v:K(X)v] = j. \]

	\parm We now assume that $(K(X)|K,v)$ is residue transcendental, that is, $\g$ is a torsion element modulo $vK$. Set
	\[ e:= (vK(X):vK(a)) \text{ and } E:= (vK(a,X): vK(a)). \]
	Hence $E = \l e$ where $\l := (vK(a,X): vK(X))$. It follows from [\ref{APZ characterization of residual trans extns}, Theorem 2.1] that $e$ is the least positive integer such that $evQ \in vK(a)$. By Lemma \ref{Lemma d = vg}, we can take $g(X)\in K[X]$ with $\deg g < \deg Q$ such that $vg = -evQ$. We can then conclude from our observations in [\ref{Dutta min fields implicit const fields}, Remark 3.1] that
	\begin{equation}\label{eqn K(X)v}
		K(X)v = K(a)v(gQ^e v).
	\end{equation}  
	Observe that $(K(a,X)|K(a),v)$ is a residue transcendental extension with $(a,\g)$ as a minimal pair of definition and the corresponding minimal polynomial $X-a$. By definition, $v(X-a) = \g$. Similar arguments as above now imply that $E$ is the smallest positive integer such that $E\g\in vK(a)$. Take $f(X)\in K[X]$ with $\deg f < \deg Q$ such that $vf = -E\g$. Then by Lemma \ref{Lemma K(a,X)v},  
	\begin{equation}\label{eqn K(a,X)v}
		K(a,X)v = K(a)v (f(X)(X-a)^Ev).
	\end{equation}
The fact that $vQ = j\g+\a$ implies that 
\[ \l e vQ = EvQ = jE\g + E\a.   \]
The fact that $evQ, E\g \in vK(a)$ then implies that $E\a\in vK(a)$. Consequently, it follows from Lemma \ref{Lemma d = vg} that there exists $h(X) \in K[X]$ with $\deg h < \deg Q$ such that $vh = -E\a$. We have thus obtained that $\l vg - jvf - vh = 0$, that is, $v\frac{g^\l}{f^j h} = 0$. It follows from [\ref{Dutta min fields implicit const fields}, Lemma 6.1] that
\[ \frac{g^\l}{f^j h}v\in K(a)v. \]
As a consequence, 
\[ K(a)v ((gQ^ev)^\l) = K(a)v(g^\l Q^Ev) = K(a)v \mathlarger{\mathlarger{(}}(\frac{g^\l}{f^jh}v)(f^jhQ^E v)\mathlarger{\mathlarger{)}} = K(a)v (f^jhQ^Ev).  \] 
We observe from [\ref{APZ characterization of residual trans extns}, Theorem 2.1] that $gQ^ev$ is transcendental over $K(a)v$. Consequently, $g^\l Q^E v = (gQ^e)^\l v$ is also transcendental over $K(a)v$. Thus,
\begin{equation}\label{eqn [K(X)v : L]}
	[K(X)v: K(a)v (f^jhQ^Ev)] = [K(a)v(gQ^ev) : K(a)v(g^\l Q^Ev)] = \l.
\end{equation}
Observe that $X-a$ divides $Q(X)$ over $K(a)$. Hence we have an expression of the form 
\[ Q^E = \sum_{i=E}^{nE} c_i (X-a)^i \text{ where } c_i \in K(a). \]
It follows from [\ref{APZ characterization of residual trans extns}, Theorem 2.1] that $v c_i + i\g \geq EvQ$ for all $i$. Suppose that $vc_i + i\g = EvQ$ for some $i$ such that $E$ does not divide $i$. Then we can write $vc_i + i^\prime\g + tE\g = EvQ $ where $1\leq i^\prime\leq E-1$. The fact that $vc_i, E\g, EvQ \in vK(a)$ then implies that $i^\prime\g\in vK(a)$, which contradicts the minimality of $E$. Hence, 
\[ E \text{ divides }i \text{ whenever } vc_i + i\g = EvQ. \]
We now consider the expression
\[ f^j h Q^E = \sum_{i=E}^{nE} f^j h c_i (X-a)^i. \]
Observe that $v f^jhQ^E = 0$. The preceding observations now imply that
\[ E \text{ divides }i \text{ whenever } v f^j h c_i (X-a)^i = 0. \] 
Taking residues, we obtain that
\[ f^jhQ^E v = \sum_{i=1}^{n} f^j h c_{iE}(X-a)^{iE}v = \sum_{i=1}^{n} (f^{j-i}hc_{iE})v (f(X)(X-a)^E)^i v. \]
It follows from [\ref{Dutta min fields implicit const fields}, Lemma 6.1] that $(f^{j-i}hc_{iE})v \in K(a)v$. As a consequence, 
\begin{equation}\label{eqn f^jhQ^Ev is a polynomial}
	f^j h Q^E v \in K(a)v [f(X)(X-a)^E v].
\end{equation}
The coefficient of $(f(X)(X-a)^E)^iv$ in $f^jhQ^Ev$ is given by $f^{j-i}hc_{iE}v$, where $c_{iE}$ is the coefficient of $(X-a)^{iE}$ in $Q^E$. Hence, 
\[ c_{iE} = (-1)^{nE-iE} \mathcal{E}_{nE-iE}(0, \dotsc, 0, a_2-a, \dotsc, a_2-a, \dotsc, a_n-a, \dotsc, a_n-a), \]
where $0$ and $a_i - a$ appear $E$ times for each $i$ and $\mathcal{E}_{nE-iE}(Y_1, \dotsc,Y_{nE})$ is the $(nE-iE)$-th elementary symmetric polynomial in the variables $Y_1, \dotsc, Y_{nE}$. By definition, each contributing term in $\mathcal{E}_{nE-iE}(0, \dotsc, 0, a_2-a, \dotsc, a_2-a, \dotsc, a_n-a, \dotsc, a_n-a)$ is of the form $(a_{t_1} - a) \dotsc (a_{t_{nE-iE}} - a)$, where $a_{t_1}, \dotsc, a_{t_{nE-iE}} \in \{a,a_2, \dotsc,a_n\}$. We first assume that $i>j$. The fact that $v(a-a_i) < \g$ for all $i>j$ now implies that
\[ v((a_{t_1} - a) \dotsc (a_{t_{nE-iE}} - a)) + (iE-jE)\g > E v((a_{j+1}-a) \dotsc (a_n-a)) = E\a. \]
Recall that $vh=-E\a$ and $vf=-E\g$. It follows that $vf^{j-i}hc_{iE} > 0$ and as a consequence,
\begin{equation}\label{eqn fhc v = 0 when i>j}
	f^{j-i}hc_{iE} v = 0 \text{ whenever } i>j. 
\end{equation}
We now assume that $i=j$. The coefficient of $(f(X)(X-a)^E)^jv$ in $f^jhQ^Ev$ is given by $hc_{jE}v$. Each contributing term in the expression of $c_{jE}$ is of the form $(a_{t_1} - a) \dotsc (a_{t_{nE-jE}} - a)$, where $a_{t_1}, \dotsc, a_{t_{nE-jE}} \in \{a,a_2, \dotsc,a_n\}$. If $\{a_{t_1}, \dotsc, a_{t_{nE-jE}} \} = \{ a_{j+1}, \dotsc, a_n  \}$ with each term appearing $E$ times, then
\[ v((a_{t_1} - a) \dotsc (a_{t_{nE-jE}} - a)) = E v((a_{j+1}-a) \dotsc (a_n-a)) = E\a = -vh. \]
Otherwise, there exists some $t_k$ such that $v(a_{t_k}-a) \geq \g$ and as a consequence, 
\[ v((a_{t_1} - a) \dotsc (a_{t_{nE-jE}} - a)) > E v((a_{j+1}-a) \dotsc (a_n-a)). \]
It now follows from the triangle inequality that $v c_{jE} = -vh$. Consequently, 
\begin{equation}\label{eqn hc v not  0}
	hc_{jE}v\neq 0.
\end{equation}
It follows from (\ref{eqn hc v not  0}) and (\ref{eqn fhc v = 0 when i>j}) that 
\begin{equation}\label{eqn deg = j}
	 \deg (f^jhQ^Ev) = j.
\end{equation}
As a consequence of (\ref{eqn f^jhQ^Ev is a polynomial}) and (\ref{eqn K(a,X)v}) we then obtain that
\begin{equation}
	[K(a,X)v: K(a)v (f^jhQ^Ev)] = [K(a)v(f(X)(X-a)^Ev) : K(a)v (f^jhQ^Ev)] = j.
\end{equation}
In light of the multiplicative property of degrees of field extensions, it now follows from (\ref{eqn [K(X)v : L]}) that $[K(a,X)v:K(X)v]\l = j$. Recall that $\l = (vK(a,X):vK(X))$. It follows that
\begin{equation}
	[K(a,X)v:K(X)v](vK(a,X):vK(X)) = j.
\end{equation}
We have thus proved the theorem.
\end{proof}


\section{Independence of $j$} \label{Sec independence of j}

Take any $a\in\overline{K}$ and $\g$ in some ordered abelian group containing $v\overline{K}$. Recall that any polynomial $f(X)\in \overline{K}[X]$ has a unique expression of the form $f(X) = \sum_{i=0}^{n} c_i (X-a)^i$, where $c_i \in \overline{K}$. Consider the map $v_{a,\g}: \overline{K}[X] \to v\overline{K} + \ZZ\g$ by setting 
\[ v_{a,\g}f:= \min\{ v c_i + i\g \}. \]
Extend $v_{a,\g}$ canonically to $\overline{K}(X)$. Then $v_{a,\g}$ is a valuation transcendental extension of $v$ from $\overline{K}$ to $\overline{K}(X)$ [\ref{Kuh value groups residue fields rational fn fields}, Lemma 3.10]. By definition, $v_{a,\g} (X-a^\prime):= \min \{ \g, v(a-a^\prime) \}$ for any $a^\prime\in\overline{K}$. It follows that 
\[ v_{a,\g}(X-a) = \g = \max v_{a,\g}(X-\overline{K}). \]

\begin{Lemma}\label{Lemma j(a,g) = j(a,G)}
	Assume that $(K(X)|K,v)$ is a residue transcendental extension. Take a minimal pair of definition $(a,\g)$ for $v$ over $K$. Consider the ordered abelian group $v\overline{K}\dirsum\ZZ$ equipped with the lexicographic order. Embed $v\overline{K}$ into $(v\overline{K}\dirsum\ZZ)_{\text{lex}}$ by setting $\a\mapsto (\a,0)$ for all $\a\in v\overline{K}$. Define
	\[ \G:= (\g,-1). \]
	Take the extension $w:= v_{a,\G}$ of $v$ to $\overline{K}(X)$. Then $(a,\G)$ is a minimal pair of definition for $w$ over $K$. Further, 
	\[ j(a,K,\g) = j(a,K,\G). \]
\end{Lemma}

\begin{proof}
	It follows from our preceding discussions that 
	\[ w(X-a) = \G = \max w(X-\overline{K}). \]
	Take any $a^\prime\in\overline{K}$. By definition, $v(a-a^\prime) \geq \G$ if and only if $(v(a-a^\prime),0) \geq (\g, -1)$, which again holds if and only if $v(a-a^\prime) \geq \g$. We have thus obtained that
	\begin{equation}\label{eqn v(a-a prime) geq gamma}
		v(a-a^\prime) \geq \G \text{ if and only if } v(a-a^\prime)\geq \g.
	\end{equation}  
	Recall that $(a,\g)$ is a minimal pair of definition for $v$ over $K$. As a consequence of (\ref{eqn v(a-a prime) geq gamma}), we now obtain that
	\[ (a,\G) \text{ is a minimal pair of definition for $w$ over $K$}. \]
	It further follows from (\ref{eqn v(a-a prime) geq gamma}) that
	\[ j(a,K,\G) = j(a,K,\g). \]
\end{proof}

\begin{Lemma}\label{Lemma j(a) = j(a prime) when v is v.t.}
	Assume that $(K(X)|K,v)$ is value transcendental. Take minimal pairs of definition $(a,\g)$ and $(a^\prime, \g)$ for $v$ over $K$. Then $j(a) = j(a^\prime)$.
\end{Lemma}

\begin{proof}
	Take the minimal polynomial $Q(X)$ of $a$ over $K$ and the minimal polynomial $Q^\prime(X)$ of $a^\prime$ over $K$. Then $vQ = j(a)\g+\a$ and $vQ^\prime = j(a^\prime)\g+\a^\prime$, where $\a,\a^\prime\in v\overline{K}$. It follows from [\ref{Dutta min fields implicit const fields}, Lemma 3.10] that $vQ = vQ^\prime$. Consequently, $(j(a)-j(a^\prime))\g \in v\overline{K}$. The fact that $\g$ is not contained in the divisible group $v\overline{K}$ implies that $j(a) = j(a^\prime)$.
\end{proof}

\pars We can now give a \textbf{proof of Theorem \ref{Thm independence of j}}.

\begin{proof}
	When the extension $(K(X)|K,v)$ is value transcendental, then the assertion of Theorem \ref{Thm independence of j} is proved in Lemma \ref{Lemma j(a) = j(a prime) when v is v.t.}. We now assume that $(K(X)|K,v)$ is residue transcendental. Consider the ordered abelian group $(v\overline{K}\dirsum\ZZ)_{\text{lex}}$. Embed $v\overline{K}$ into $(v\overline{K}\dirsum\ZZ)_{\text{lex}}$ by setting $\a\mapsto (\a,0)$ for all $\a\in v\overline{K}$. Set $\G:= (\g,-1)$. The fact that $(a,\g)$ and $(a^\prime,\g)$ are minimal pair of definition for $v$ over $K$ implies that $v(a-a^\prime) \geq \g$ and hence $v(a-a^\prime) \geq \G$. It now follows from [\ref{AP sur une classe}, Proposition 3] that $v_{a,\G} = v_{a^\prime,\G}$. Set $w:= v_{a,\G} = v_{a^\prime,\G}$. In light of Lemma \ref{Lemma j(a,g) = j(a,G)}, we observe that $(a,\G)$ and $(a^\prime,\G)$ are minimal pairs of definition of $w$ over $K$. Further, $j(a,K,\g) = j(a,K,\G)$ and $j(a^\prime,K,\g) = j(a^\prime,K,\G)$. The fact that $\G\notin v\overline{K}$ implies that $w$ is a value transcendental extension of $v$ to $\overline{K}(X)$. It then follows from Lemma \ref{Lemma j(a) = j(a prime) when v is v.t.} that $j(a,K,\G) = j(a^\prime,K,\G)$. As a consequence, we conclude that
	\[ j(a,K,\g) = j(a^\prime,K,\g). \] 
\end{proof}


\section{Implicit constant fields}\label{Sec impl const fields}

\begin{Proposition}\label{Prop K rel alg closed in L}
	Assume that $L|K$ is an extension of fields such that $K$ is relatively algebraically closed in $L$. Take $a\in\overline{K}$. Then $K(a)$ and $L$ are linearly disjoint over $K$.
\end{Proposition}

\begin{proof}
	Take the minimal polynomial $Q(X)$ of $a$ over $L$. Write 
	\[ Q(X) = (X-a)(X-a_2) \dotsc (X-a_n) = \sum_{i=0}^{n} c_i X^i. \]
	Take the minimal polynomial $f(X)$ of $a$ over $K$. Then $Q$ divides $f$ over $L$ and hence each root of $Q$ is also a root of $f$. Thus each $a_i$ is a $K$-conjugate of $a$ and consequently $a_i \in \overline{K}$. Further, observe that each coefficient $c_j$ is a symmetric expression in the roots $a_i$ and hence $c_j \in \overline{K}$ for all $j$. Consequently, $c_j \in \overline{K}\sect L = K$, that is, $Q(X)\in K[X]$. It follows that 
	\[ [K(a):K] = [L(a):L], \]
	that is, $K(a)$ and $L$ are linearly disjoint over $K$. 
\end{proof}

\textit{For the rest of this section}, we fix an extension of $v$ to $\overline{K(X)}$. Take a minimal pair of definition $(a,\g)$ for $v$ over $K$. We have observed in [\ref{Dutta min fields implicit const fields}, Theorem 1.1] that $K^h \subseteq IC(K(X)|K,v) \subseteq K(a)^h$. As a consequence, 
\[ IC(K(X)|K,v)(a) = K(a)^h. \]
By definition, $IC(K(X)|K,v)$ is relatively algebraically closed in $K(X)^h$. It is now a direct consequence of Proposition \ref{Prop K rel alg closed in L} that $K(a)^h$ and $K(X)^h$ are linearly disjoint over $IC(K(X)|K,v)$. Observe that $K(a)^h.K(X)^h = K(X)^h(a) = K(a,X)^h$. Thus,
\[ [K(a)^h:IC(K(X)|K,v)] = [K(a,X)^h: K(X)^h]. \]
From the Lemma of Ostrowski, we have that 
\[ [K(a,X)^h: K(X)^h] = (vK(a,X)^h:vK(X)^h)[K(a,X)^hv: K(X)^hv] d(K(a,X)^h|K(X)^h, v). \]
Recall that henselization is an immediate extension. The following result now follows immediately from Theorem \ref{Thm central}: 
\begin{equation}\label{eqn j(a) divides deg K(a)h over IC}
	[K(a)^h:IC(K(X)|K,v)] = j(a)d(K(a,X)^h|K(X)^h, v).
\end{equation}
	Assume that $(K,v)$ is a defectless valued field. The fact that $(K(X)|K,v)$ is a valuation transcendental extension implies that we can apply [\ref{Kuh gen stab thm}, Theorem 1.1] to this extension and obtain that $(K(X),v)$ is also a defectless field. By definition, $(K(X)^h, v)$ is a defectless field. We have thus arrived at the following result:

\begin{Theorem}\label{Thm j(a) divides deg K(a)h over IC}
	Take a minimal pair of definition $(a,\g)$ for $v$ over $K$. Fix an extension of $v$ to $\overline{K(X)}$. Then
	\[ [K(a)^h:IC(K(X)|K,v)] = j(a)d(K(a,X)^h|K(X)^h, v). \]
	In particular, 
	\[ [K(a)^h:IC(K(X)|K,v)] = j(a) \text{ whenever } (K,v) \text{ is defectless}. \]
\end{Theorem} 

\begin{Corollary}
	Take a minimal pair of definition $(a,\g)$ for $v$ over $K$. Fix an extension of $v$ to $\overline{K(X)}$. 
	\sn (i) A necessary condition for obtaining $IC(K(X)|K,v) = K(a)^h$ is that $j(a)=1$. If $(K,v)$ is defectless, then the condition is also sufficient.
	\n (ii) Assume that $j(a) = [K(a):K]$. Then $IC(K(X)|K,v) = K^h$. 
\end{Corollary}

\begin{proof}
	The first assertion is an immediate consequence of Theorem \ref{Thm j(a) divides deg K(a)h over IC}. It is thus enough to prove $(ii)$. We assume that $j(a) = [K(a):K]$. It then follows from Theorem \ref{Thm j(a) divides deg K(a)h over IC} that
	\[ [K(a)^h:IC(K(X)|K,v)] \geq j(a) = [K(a):K]. \]
	Recall that $K^h(a) = K(a)^h$. Consequently, 
	\[ [K(a):K] \geq [K(a)^h:K^h] \geq [K(a)^h:IC(K(X)|K,v)]. \]
	It now follows that $[K(a)^h:IC(K(X)|K,v)] = [K(a)^h:K^h]$. As a consequence, 
	\[ IC(K(X)|K,v) = K^h. \]
\end{proof}

An immediate corollary is the following:

\begin{Corollary}([\ref{Dutta min fields implicit const fields}, Proposition 7.1])
	Take a minimal pair of definition $(a,\g)$ for $v$ over $K$. Assume that $a$ is purely inseparable over $K$. Fix an extension of $v$ to $\overline{K(X)}$. Then $IC(K(X)|K,v) = K^h$.
\end{Corollary}


\section{Proof of Theorem \ref{Thm IC bounds}} \label{Sec proof of thm 1.2}

\begin{proof}
	When the  extension $(K(X)|K,v)$ is value transcendental, the assertion is proved in [\ref{Dutta min fields implicit const fields}, Theorem 1.1]. So we assume that $(K(X)|K,v)$ is residue transcendental.
	
	\pars We fix an extension of $v$ to $\overline{K(X)}$ and denote it again by $v$. Observe that $K(a)\sect K^r$ is a finite separable extension, hence simple. Hence we can choose $b\in \overline{K}$ such that $K(a)\sect K^r = K(b)$. The inclusion $IC(K(X)|K,v) \subseteq K(a)^h$ follows from [\ref{Dutta min fields implicit const fields}, Lemma 5.1]. It is left to show that $K(b)^h \subseteq IC(K(X)|K,v)$. Recall that $IC(K(X)|K,v):= \overline{K}\sect K(X)^{h(v)}$. It is thus sufficient to show that $b\in K(X)^{h(v)}$, that is, $K(b,X)^{h(v)} = K(X)^{h(v)}$. 
	
\pars	Applying the Lemma of Ostrowski to the extension $(K(a,X)^{h(v)}|K(X)^{h(v)}, v)$, we obtain the following relation in light of Theorem \ref{Thm central}:
	\[ [K(a,X)^{h(v)} : K(X)^{h(v)}] = j(a,K,\g) d(K(a,X)^{h(v)}|K(X)^{h(v)}, v). \]
Take $a^\prime\in\overline{K}$ such that $v(a-a^\prime) \geq\g$. By definition, $[K(a):K] \leq [K(a^\prime):K]$. The fact that $b\in K(a)$ then implies that 
\[ [K(a,b):K] = [K(a):K] \leq [K(a^\prime):K] \leq [K(a^\prime,b):K]. \]
Consequently, 
\[ [K(a,b):K(b)] \leq [K(a^\prime,b):K(b)]. \]
It follows that $(a,\g)$ is also a minimal pair of definition for $v$ over $K(b)$. It now follows from Theorem \ref{Thm central} and the Lemma of Ostrowski that
\[ [K(a,X)^{h(v)} : K(b,X)^{h(v)}] = j(a,K(b),\g) d(K(a,X)^{h(v)}|K(b,X)^{h(v)}, v). \]
As a consequence, 
\[ [K(b,X)^{h(v)} : K(X)^{h(v)}] = \frac{j(a,K,\g)}{j(a,K(b),\g)} d(K(b,X)^{h(v)}|K(X)^{h(v)}, v). \]
It follows from [\ref{Dutta Kuh abhyankars lemma}, Theorem 3] that the condition that $b\in K^r$ implies that $K(b,X) \subseteq K(X)^{r(v)}$. Consequently, $(K(b,X)^{h(v)}|K(X)^{h(v)}, v)$ is a defectless extension. We have thus obtained that
\begin{equation}\label{eqn K(b,X)h(v) over K(X)h(v)}
	[K(b,X)^{h(v)} : K(X)^{h(v)}] = \frac{j(a,K,\g)}{j(a,K(b),\g)}.
\end{equation}

\parm We now consider the valuation $w$ as constructed in the statement of Lemma \ref{Lemma j(a,g) = j(a,G)}. Fix an extension of $w$ to $\overline{K(X)}$ and denote it again by $w$. Similar arguments as above yield that
\[ [K(b,X)^{h(w)} : K(X)^{h(w)}] = \frac{j(a,K,\G)}{j(a,K(b),\G)}. \]
Observe that $\G\notin v\overline{K}$ and hence $(K(X)|K,w)$ is a value transcendental extension. It then follows from [\ref{Dutta min fields implicit const fields}, Theorem 1.1] that $K(b)^h \subseteq IC(K(X)|K,w)$. As a consequence, $b\in K(X)^{h(w)}$ and hence $K(b,X)^{h(w)} = K(X)^{h(w)}$. In light of the preceding discussions, we conclude that
\[ j(a,K,\G) = j(a,K(b),\G). \]
It follows from Lemma \ref{Lemma j(a,g) = j(a,G)} that $j(a,K,\g) = j(a,K,\G)$. Observe that $(a,\g)$ is also a minimal pair of definition for $v$ over $K(b)$. Applying Lemma \ref{Lemma j(a,g) = j(a,G)} to the extension $(K(b,X)|K(b),v)$, we obtain that $j(a,K(b),\G) = j(a,K(b),\g)$. As a consequence, 
\[ j(a,K,\g) = j(a,K,\G) = j(a,K(b),\G) = j(a,K(b),\g). \]
It now follows from (\ref{eqn K(b,X)h(v) over K(X)h(v)}) that $K(b,X)^{h(v)} = K(X)^{h(v)}$. Consequently, $b\in K(X)^{h(v)}$ and hence $K(b)^h \subseteq IC(K(X)|K,v)$. We have thus proved the theorem. 
\end{proof}

\begin{Corollary}
	Take a minimal pair of definition $(a,\g)$ for $v$ over $K$ and fix an extension of $v$ to $\overline{K(X)}$. Assume that $j(a) = [K(a):K(a)\sect K^r]$. Then $IC(K(X)|K,v) = (K(a)\sect K^r)^h$.
\end{Corollary}

\begin{proof}
	Take $b\in\overline{K}$ such that $K(b) = K(a)\sect K^r$. Recall that $j(a)$ divides $[K(a)^h:IC(K(X)|K,v)]$. In light of Theorem \ref{Thm IC bounds} we have the following chain of relations:
	\[ j(a) \leq [K(a)^h:IC(K(X)|K,v)] \leq [K(a)^h:K(b)^h] \leq [K(a):K(b)]. \]
	The assumption that $j(a)= [K(a):K(b)]$ now implies that each of the inequalities in the above expression is an equality. As a consequence, we obtain that 
	\[ IC(K(X)|K,v) = K(b)^h. \]
\end{proof}


\end{document}